\documentclass{amsart}

\newtheorem{theorem}{Theorem}[section]
\newtheorem{lemma}[theorem]{Lemma}

\newtheorem{corollary}[theorem]{Corollary}

\newtheorem{problem}[theorem]{Problem}

\theoremstyle{definition}

\newtheorem{example}[theorem]{Example}

\newtheorem{remark}[theorem]{Remark}

\usepackage{amscd,amssymb}

\begin{document}

\title[Symmetric polynomials in the variety generated by Grassmann algebras]
{Symmetric polynomials in the variety \\ generated by Grassmann algebras}
\author[Nazan Akdo\u{g}an, {\c S}ehmus F{\i}nd{\i}k]
{Nazan Akdo\u{g}an, {\c S}ehmus F{\i}nd{\i}k}
\address{Department of Mathematics,
\c{C}ukurova University, 01330 Balcal\i,
 Adana, Turkey}
\address{Department of Mathematics,
\.Istanbul Technical University, \.Istanbul, Turkey}
\email{nakdogan@itu.edu.tr}
\address{Department of Mathematics,
\c{C}ukurova University, 01330 Balcal\i,
 Adana, Turkey}
\email{sfindik@cu.edu.tr}


\begin{abstract}
Let $\mathcal{G}$ denote the variety generated by infinite dimensional Grassmann algebras; i.e.,
the collection of all unitary associative algebras satisfying the identity $[[z_1,z_2],z_3]=0$, where $[z_i,z_j]=z_iz_j-z_jz_i$.
Consider the free algebra $F_3$ in $\mathcal{G}$ generated by $X_3=\{x_1,x_2,x_3\}$.
We call a polynomial $p\in F_3$ {\it symmetric} if it is preserved under the action of the symmetric group $S_3$ on generators, i.e., $p(x_1,x_2,x_3)=p(x_{\xi1},x_{\xi2},x_{\xi3})$
for each permutation $\xi\in S_3$. The set of symmetric polynomials form the subalgebra $F_3^{S_3}$ of invariants of the group $S_3$ in $F_3$.
The commutator ideal $F_3'$ of the algebra $F_3$ has a natural left $K[X_3]$-module structure, and $(F_3')^{S_3}$
is a left $K[X_3]^{S_3}$-module.
We give a finite free generating set for the $K[X_3]^{S_3}$-module $(F_3')^{S_3}$. 
\end{abstract}

 \maketitle

\section{Introduction}

Let $K[X_n]$ be the commutative associative unitary polynomial algebra over a field $K$ of  characteristic zero generated by $X_n=\{x_1,\dots,x_n\}$. We consider the general linear group $\text{GL}_n(K)$ which acts on $x_1,\dots,x_n$. This action is extended to an action of $\text{GL}_n(K)$ on $K[X_n]$ canonically, which is an automorphism of the polynomial algebra $K[X_n]$. If $G\leq \text{GL}_n(K)$, then the algebra $K[X_n]^G$ of invariants of $G$ is defined to be the set of polynomials $p(X_n)\in K[X_n]$ such that $g(p(X_n))=p(X_n)$ for each $g\in G$.

One may see the works of Hilbert (\cite{H},\cite{H2}) published in the last years of $19^{th}$ century, dealing with the finite generation of the algebra of invariants and
finite bases of syzygies, which are known as the first and the second fundamental theorems of invariant theory.
At the 1900 Paris congress, {\it Hilbert's $14^{th}$ problem} \cite{H3} was proposed by himself to generalize his result to any subgroup of $\text{GL}_n(K)$. A special case of the problem is to determine whether the algebra $K[X_n]^G$ of invariants is finitely generated for any subgroup $G$ of $\text{GL}_n(K)$. 
Some partial affirmative results were obtained in response to the problem.
For example, Noether \cite{NE} showed that the algebra $K[X_n]^G$ of invariants is finitely generated when $G$ is a finite subgroup of $GL_n(K)$.
Another approach in the same direction was given by Shephard and Todd \cite{S} and Chevalley \cite{C} which considers a special case, when the finite group $G$ is generated by reflections, and $K[X_n]^G$ is generated by $n$ algebraically independent homogeneous elements.
Finally, Nagata \cite{N} gave a counterexample to Hilbert's $14^{th}$ problem in 1959.  
One can see the survey on this problem by Humpreys \cite{HJE} for more information.

The problem were applied to noncommutative invariant theory, as well. Let $K\langle X_n \rangle$ denote a free associative algebra.
Kharchenko \cite{K2} extended Noether's result proving that $(K\langle X_n \rangle / T)^G$, where $T$ is a $T$-ideal, is finitely generated for every finite subgroup $G$ of $\text{GL}_n(K)$ if and only if $K\langle X_n \rangle / T$ satisfies the ascending chain condition on two-sided ideals. Furthermore, Dicks and Formanek \cite{D} and Kharchenko \cite{K2} identified the condition on finite groups $G$ for finite generation of $K\langle X_n \rangle ^G$: The algebra of invariants $K\langle X_n \rangle ^G$ of a finite subgroup $G$ of $\text{GL}_n(K)$ is finitely generated if and only if $G$ is a cyclic group acting by scalar multiplication. Lane \cite{L} and Kharchenko \cite{K} showed that $K\langle X_n \rangle ^G$ is a free algebra for any finite subgroup of $\text{GL}_n(K)$. Long after, Domokos \cite{DM} showed that $(K\langle X_n \rangle/T)^G$ is generated by algebraically independent elements for a finite group $G$ if and only if $G$ is a pseudo-reflection group and $T$ contains the polynomial $[[z_1,z_2],z_2]$. We recommend Formanek's survey \cite{FE} for more detail on the noncommutative generalizations of the problem. 

The most basic example of an algebra of invariants is defined by the symmetric polynomials.
Let $G=S_n$ be the group of permutation matrices in $\text{GL}_n(K)$ acting on the algebra $K[X_n]$ by permuting  the variables; i.e.,
\[
\xi\cdot p(x_1,\ldots,x_n)=p(x_{\xi1},\ldots,x_{\xi n}) \ , \ \ \xi\in S_n \ , \ \ p\in K[X_n].
\]
The algebra $K[X_n]^{S_n}$ of invariants of $S_n$ is a subalgebra of symmetric polynomials in $K[X_n]$. It is well known that $K[X_n]^{S_n}$ is generated by the elementary symmetric polynomials $\sigma_j=\sum x_{i_1}\cdots x_{i_j}$, $i_1<\cdots<i_j$, $j=1,\dots,n$ which are algebraically independent over $K$. Furthermore, $K[X_n]^{S_n}$ is isomorphic to a polynomial algebra in $n$ variables since $\sigma_1,\dots,\sigma_n$ are algebraically independent. Besides, the set $\{\sum_{i=1}^n x_i^k : k=1,\dots,n \}$ is another generating set for $K[X_n]^{S_n}$.

When it comes to noncommutative-nonassociative algebras, one may consider Lie algebras, Leibniz algebras, etc. Bryant and Papistas \cite{B} 
showed that the algebra of invariants in the free Lie algebra is not finitely generated for a finite nontrivial group of automorphisms of the algebra.
A similar result is obtained for free metabelian Lie algebra in the same paper. Some recent works related with the Lie and Leibniz generalizations of the problem have been carried out. Let $L_n$ and $\mathcal L_n$ be the free metabelian Lie and Leibniz algebras
of rank $n$, respectively. In both cases the commutator ideals $L_n'$ and $\mathcal L_n'$ of $L_n$ and $\mathcal L_n$, respectively, are right $K[X_n]$-modules. A minimal infinite generating set for the Lie algebra
$L_2^{S_2}$ was given in \cite{F}, while a finite generating set for the $K[X_n]^{S_n}$-module $(L_n')^{S_n}$ was provided by Drensky et al. \cite{DV2}. Besides, characterization of the symmetric polynomials in $(\mathcal L_n')^{S_n}$ as a $K[X_n]^{S_n}$-module was given in \cite{FO}.

An important class of noncommutative algebras is the variety generated by the Grassmann identity $[[z_1,z_2],z_3]=0$. Let $F_n$ be the free algebra of rank $n$ in this variety.
The symmetric polynomials in $F_2^{S_2}$ have been studied in \cite{Na}.
In this paper, we consider the free associative algebra $F_3$ of rank $3$ generated by $X_3$ over a field $K$ of characeristic zero. Its commutator ideal $F_3'=F_3[F_3,F_3]F_3$
is a left $K[X_3]$-module. We investigate the algebra $F_3^{S_3}$ of invariants of ${S_3}$,
and give a finite free generating set for $(F_3')^{S_3}$ as a left $K[X_3]^{S_3}$-module.

\section{Preliminaries}

Let $K$ be a field of characteristic zero and let $K\langle Y_n\rangle$ be the free associative algebra over $K$ with the generating set $Y_n=\{y_1,\dots,y_n\}$.
An element $f$ in the algebra $K\langle Y_n\rangle$ is a polynomial identity of an algebra $A$ if $f(a_1,\dots,a_n)=0$ for all $a_1,\dots,a_n\in A$. An algebra $A$ is called a {\it polynomial identity algebra} or simply a $\text{PI}$-{\it algebra} if it has a nontrivial polynomial identity. The set $T(A)$ of all polynomial identities of $A$, which are invariant under all endomorphisms of $K\langle Y_n\rangle$, forms an ideal of $K\langle Y_n\rangle$. This ideal is called the $T$-ideal of $A$. The class of all associative algebras satisfying the polynomial identities from $T(A)$ is the variety generated by the algebra $A$.

Now let $I$ be the ideal of $K\langle Y_n\rangle$ generated by the elements of the form $y_iy_j+y_jy_i$, $1\leq i,j\leq n$. The Grassmann algebra is the factor algebra $K\langle Y_n\rangle / I$, generated by the elements $y_i+I$, $1\leq i \leq n$, over the field $K$. The Grassmann algebra is a $\text{PI}$-algebra satisfying the polynomial identity $[[z_1,z_2],z_3]=0$ , where $[z_1,z_2]=z_1z_2-z_2z_1$. Krakowski and Regev \cite{KD} proved that the $T$-ideal of the Grassmann algebra is generated by this polynomial identity. 
See also \cite{GK} on the identities of the Grassmann algebras when the based field is of positive characteristic, \cite{Cen} for graded identities of the Grassmann algebra,
and \cite{M} on Grassmann algebras of graphs.

The variety $\mathcal{G}$ generated by the Grassmann algebra contains unitary associative algebras
satisfying the polynomial identity
\[
[[z_1,z_2],z_3]=[z_1,z_2]z_3-z_3[z_1,z_2]=(z_1z_2-z_2z_1)z_3-z_3(z_1z_2-z_2z_1)=0,
\]
over the field $K$. Recall that the above identity implies
\begin{equation}\label{identityinparticular}
[z_1,z_2][z_3,z_4]=-[z_1,z_3][z_2,z_4].
\end{equation}

\

Let $F_n=F_n(\mathcal{G})$ be the free algebra in $\mathcal{G}$ of rank $n$ 
generated by $X_n=\{x_1,\ldots,x_n\}$ over $K$.
The commutator ideal $F_n'=F_n[F_n,F_n]F_n$
of the algebra $F_n$ has the following basis via \cite{DV}.
\[
x_1^{a_1}\cdots x_{n}^{a_n}[x_{i_1},x_{i_2}]\cdots[x_{i_{2c-1}},x_{i_{2c}}], \ a_i \geq0,\  i_1 >\cdots>i_{2c},\ c\geq1.
\]
The identity (\ref{identityinparticular}) implies that $[x_{i_1},x_{i_2}][x_{i_3},x_{i_4}]=0$ if some of $x_{i_j}$'s coincide, since
\[
[x_{i_1},x_{i_2}][x_{i_3},x_{i_4}]=-[x_{i_1},x_{i_3}][x_{i_2},x_{i_4}]=[x_{i_1},x_{i_4}][x_{i_2},x_{i_3}]
\]
and $[x_{i_j},x_{i_j}]=x_{i_j}^2-x_{i_j}^2=0$. This means that
\begin{equation}\label{theorem}
[x_2,x_1][x_2,x_1]=[x_3,x_1][x_2,x_1]=[x_3,x_2][x_2,x_1]=0
\end{equation}
\[
[x_2,x_1][x_3,x_1]=[x_3,x_1][x_3,x_1]=[x_3,x_2][x_3,x_1]=0
\]
\[
[x_2,x_1][x_3,x_2]=[x_3,x_1][x_3,x_2]=[x_3,x_2][x_3,x_2]=0
\]
or
\[
(x_2x_1-x_1x_2)[x_2,x_1]=(x_3x_1-x_1x_3)[x_2,x_1]=(x_3x_2-x_2x_3)[x_2,x_1]=0
\]
\[
(x_2x_1-x_1x_2)[x_3,x_1]=(x_3x_1-x_1x_3)[x_3,x_1]=(x_3x_2-x_2x_3)[x_3,x_1]=0
\]
\[
(x_2x_1-x_1x_2)[x_3,x_2]=(x_3x_1-x_1x_3)[x_3,x_2]=(x_3x_2-x_2x_3)[x_3,x_2]=0
\]
Hence we have the followings in $F_2'$ and $F_3'$, respectively, as a consequence of Grassmann identity.
\[
x_1x_2w=x_2x_1w, \ \ w\in F_2' \ , 
\]
\[
x_1x_2x_3w=x_1x_3x_2w=x_2x_1x_3w=x_2x_3x_1w=x_3x_1x_2w=x_3x_2x_1w, \ \ w\in F_3'.
\]
Therefore, $F_2'$ and $F_3'$ can be considered as a left $K[X_2]$-module and $K[X_3]$-module, respectively.
Note that $F_n'$ is not a left $K[X_n]$-module when $n\geq4$, since
\[
x_1x_2[x_3,x_4]\neq x_2x_1[x_3,x_4]
\]
as a consequence of the fact that $[x_1,x_2][x_3,x_4]\neq0$.
The following result is a direct consequence of the basis and equalities provided above.

\begin{corollary}\label{free}
The commutator ideal $F_2'$ is a free left $K[X_2]$-module generated by $[x_2,x_1]$, and $F_3'$ is a free left $K[X_3]$-module 
 with generators $[x_2,x_1]$, $[x_3,x_1]$, and $[x_3,x_2]$.
\end{corollary}

A polynomial $s(X_n)\in F_n$ is called symmetric if $s(x_1,\ldots,x_n)=s(x_{\xi1},\ldots,x_{\xi n})$ for all $\xi\in S_n$.
Symmetric polynomials form a subalgebra $F_n^{S_n}$ of $F_n$ called the algebra of invariants of the symmetric group $S_n$.
Clearly the algebra $(F_2')^{S_2}$ and $(F_3')^{S_3}$ can be considered as a left $K[X_2]^{S_2}$-module and a left $K[X_3]^{S_3}$-module, respectively.

Let us initiate by $n=2$. The next theorem gives the description of the $K[X_2]^{S_2}$-module $(F_2')^{S_2}$. 
Although stated in \cite{Na}, we give a short sketch of the proof for the completeness of our paper.

\begin{theorem}\label{2}
The algebra $(F_2')^{S_2}$ of symmetric polynomials in the commutator ideal $F_2'$ is generated by the single element
$(x_2-x_1)[x_2,x_1]$ as a left $K[X_2]^{S_2}$-module.
\end{theorem}

\begin{proof}
Let $e(x_1,x_2)=p(x_1,x_2)[x_2,x_1]\in (F_2')^{S_2}$, where $p(x_1,x_2)\in K[X_2]$. Then $e(x_1,x_2)=e(x_2,x_1)$
implies that $(p(x_1,x_2)+p(x_2,x_1))[x_2,x_1]=0$ 
in the free module generated by $[x_2,x_1]$. Thus the polynomial $p(x_1,x_2)\in K[X_2]$
can be expressed as $q(x_1,x_2)(x_2-x_1)$ for some $q(x_1,x_2)\in K[X_2]^{S_2}$.
\end{proof}

The results obtained in this paper can be considered as a next step of the work done in Theorem \ref{2}.

\section{Main Results}

In this section, we study the algebra $(F_3')^{S_3}$ of symmetric polynomials in the commutator ideal $F_3'$
of the algebra $F_3$. We give explicit form of symmetric polynomials in the algebra $(F_3')^{S_3}$.
We also provide a finite set of free generators for $(F_3')^{S_3}$ as a module of
$K[X_3]^{S_3}$.
In order to study the algebra $(F_3')^{S_3}$, we work in the commutator ideal $F_3'$,
which is freely generated by $[x_2,x_1]$, $[x_3,x_1]$, and $[x_3,x_2]$ as a $K[X_3]$-module due to Corollary \ref{free}.

In the next theorem, we provide the description of $(F_3')^{S_3}$ as a left $K[X_3]^{S_3}$-module.

\begin{theorem}\label{forms}
A polynomial $f\in F_3'$ is symmetric if and only if 
\[
f(x_1,x_2,x_3)=p(x_1,x_2,x_3)[x_2,x_1]+p(x_1,x_3,x_2)[x_3,x_1]+p(x_2,x_3,x_1)[x_3,x_2]
\]
for some $p\in K[X_3]$ such that $p(x_1,x_2,x_3)=-p(x_2,x_1,x_3)$.
\end{theorem}

\begin{proof} Let $f(x_1,x_2,x_3)$ be an arbitrary element in $(F_3')^{S_3}\subset F_3'$ of the form
\[
f(x_1,x_2,x_3)=p(x_1,x_2,x_3)[x_2,x_1]+q(x_1,x_2,x_3)[x_3,x_1]+r(x_1,x_2,x_3)[x_3,x_2],
\]
for some $p,q,r\in K[X_3]$. The fact that
\[
f(x_1,x_2,x_3)=f(x_{\xi1},x_{\xi2},x_{\xi3}), \ \ \xi\in S_3,
\]
combined by Corollary \ref{free} imply the followings.
	\begin{align*}
	p(x_1,x_2,x_3)&=q(x_1,x_3,x_2)=-p(x_2,x_1,x_3)=-q(x_2,x_3,x_1)\\&=r(x_3,x_1,x_2)=-r(x_3,x_2,x_1).
	\end{align*}
One may reduce the equalities above to 
\[
q(x_1,x_2,x_3)=p(x_1,x_3,x_2) \ , \ \ r(x_1,x_2,x_3)=p(x_2,x_3,x_1)
\]
such that $p(x_1,x_2,x_3)=-p(x_2,x_1,x_3)$.

\

\noindent Conversely, let $f(x_1,x_2,x_3)$ satisfy the statement of the theorem, i.e.,
\[
f(x_1,x_2,x_3)=p(x_1,x_2,x_3)[x_2,x_1]+p(x_1,x_3,x_2)[x_3,x_1]+p(x_2,x_3,x_1)[x_3,x_2]
\]
for some $p\in K[X_3]$ such that $p(x_1,x_2,x_3)=-p(x_2,x_1,x_3)$. We have to show that
$f(x_1,x_2,x_3)$ is preserved under the action of transpositions $(12)$ and $(13)$ since they generate the symmetric group $S_3$.
We start with the action of $(12)$.
\begin{align}
f(x_2,x_1,x_3)=&p(x_2,x_1,x_3)[x_1,x_2]+p(x_2,x_3,x_1)[x_3,x_2]+p(x_1,x_3,x_2)[x_3,x_1]\nonumber\\
=&-p(x_1,x_2,x_3)[x_1,x_2]+p(x_1,x_3,x_2)[x_3,x_1]+p(x_2,x_3,x_1)[x_3,x_2]\nonumber\\
=&p(x_1,x_2,x_3)[x_2,x_1]+p(x_1,x_3,x_2)[x_3,x_1]+p(x_2,x_3,x_1)[x_3,x_2]\nonumber
\end{align}
Similarly, one may easily show that $f(x_3,x_2,x_1)=f(x_1,x_2,x_3)$.
\end{proof}	

\begin{remark}
Let $f$ be a symmetric polynomial in $(F_3')^{S_3}$ as stated in Theorem \ref{forms}. Then one may easily verify that $p=p_1p_2$ for some $p_1\in K[X_3]^{S_3}$ and some $p_2\in K[X_3]$ of the form
\[
p_2(x_1,x_2,x_3)=\sum_{0\leq a<b, \ 0\leq c}\varepsilon_{abc}(x_1^ax_2^b-x_1^bx_2^a)x_3^c
\]
where $\varepsilon_{abc}\in K$. Let us denote
\[
f_{a,b,c}=(x_1^ax_2^b-x_1^bx_2^a)x_3^c[x_2,x_1]+(x_1^ax_3^b-x_1^bx_3^a)x_2^c[x_3,x_1]+(x_2^ax_3^b-x_2^bx_3^a)x_1^c[x_3,x_2]
\]
for nonnegative integers $a,b,c$. This immediately implies the following result.
\end{remark}

\begin{corollary}\label{linear}
The set $\{f_{a,b,c} \mid 0\leq a<b, 0\leq c\}$
is a basis for the $K[X_3]^{S_3}$-module $(F_3')^{S_3}$; i.e., $f\in(F_3')^{S_3}$ if and only if
\[
f(x_1,x_2,x_3)=\sum_{0\leq a<b, \ 0\leq c} \delta_{abc}f_{a,b,c}
\]
for some uniquely determined polynomials $\delta_{abc}\in K[X_3]^{S_3}$.
\end{corollary}

The rest of the paper is devoted to obtain a generating set for the $K[X_3]^{S_3}$-module $(F_3')^{S_3}$.
Let us denote
\[
\sigma_1=x_1+x_2+x_3, \ \sigma_2=x_1x_2+x_1x_3+x_2x_3, \ \sigma_3=x_1x_2x_3,
\]
which are the elementary symmetric polynomials freely generating the algebra $K[X_3]^{S_3}$. Note that $\{\nu_1,\nu_2,\nu_3\}$
is another generating set for the algebra $K[X_3]^{S_3}$, where $\nu_k$ is defined to be $\nu_k=x_1^k+x_2^k+x_3^k\in K[X_3]^{S_3}$, for each $k\geq1$.

The proof of the following technical lemma is straightforward.

\begin{lemma}\label{tech}
\begin{align}\label{0bc}
f_{0,b,c}&=\nu_cf_{0,b,0}-f_{0,b+c,0}+f_{b,c,0} , \ \ b,c\geq1,
\end{align}
\begin{align}\label{abba}
f_{a,b,c}=-f_{b,a,c} , \ \ a,b,c\geq0,
\end{align}
\begin{align}\label{0b0}
f_{0,b,0}&=\sigma_1f_{0,b-1,0}-\sigma_2f_{0,b-2,0}+\sigma_3f_{0,b-3,0} , \ \ b\geq4
\end{align}
\begin{align}\label{030}
f_{0,3,0}=\sigma_1f_{0,2,0}-\sigma_2f_{0,1,0} ,
\end{align}
\begin{align}\label{ab0}
f_{a,b,0}&=\sigma_1f_{a,b-1,0}-\sigma_2f_{a,b-2,0}+\sigma_3f_{a,b-3,0} , \ \ a\geq1, b\geq4
\end{align}
\begin{align}\label{a30}
f_{a,3,0}=\sigma_1f_{a,2,0}-\sigma_2f_{a,1,0}-\sigma_3f_{0,a,0} , \ \ a\geq1,
\end{align}
\end{lemma}

\

The next two results are main theorems of the paper.

\begin{theorem}\label{main}
$(F'_3)^{S_3}$ is generated by the set $\{f_{0,1,0},f_{0,2,0},f_{1,2,0}\}$ as a left $K[X_3]^{S_3}$-module.
\end{theorem}

\begin{proof}
The first observation by Corollary \ref{linear} is that the elements of the form $f_{a,b,c}$ such that
$0\leq a<b, 0\leq c$ generate the $K[X_3]^{S_3}$-module $(F'_3)^{S_3}$. Depending on the relations between $a$ and $c$, we may handle the problem in three cases.

\noindent Case I. Let $a=c$, and $b=a+k$ for some positive integer $k$. Then
\begin{align}
f_{a,b,c}&=(x_1^ax_2^b-x_1^bx_2^a)x_3^a[x_2,x_1]+(x_1^ax_3^b-x_1^bx_3^a)x_2^a[x_3,x_1]+(x_2^ax_3^b-x_2^bx_3^a)x_1^a[x_3,x_2]\nonumber\\
&=(x_1x_2x_3)^a\left((x_2^k-x_1^k)[x_2,x_1]+(x_3^k-x_1^k)[x_3,x_1]+(x_3^k-x_2^k)[x_3,x_2]\right)\nonumber\\
&=\sigma_3^af_{0,k,0}\nonumber
\end{align}

\noindent Case II. Let $a<c$, and $b=a+k$, $c=a+l$, for some $k,l\in\mathbb{Z}^+$. Then similarly, we get that 
$f_{a,b,c}=\sigma_3^af_{0,k,l}$.

\noindent Case III. Let $c<a$, and $a=c+k$, $b=c+l$, for some $k,l\in\mathbb{Z}^+$. Then we have
$f_{a,b,c}=\sigma_3^cf_{k,l,0}$.

Therefore, the generating set reduces to the set consisting of the elements of the form $f_{0,b,0}$, $f_{a,b,0}$, and $f_{0,b,c}$.
Now by Lemma \ref{tech} options (\ref{0bc}) and  (\ref{abba}), we may eliminate the elements $f_{0,b,c}$ from the generating set, as well.
Note that we use the option (\ref{abba}) in case $c<b$, in order to fix the notation as in Corollary \ref{linear}.

Considering Lemma \ref{tech} options (\ref{0b0}) and  (\ref{030}), one may conclude by induction that every element of the form $f_{0,b,0}$, $b\geq3$,
is included in the $K[X_3]^{S_3}$-module generated by $f_{0,1,0}$ and $f_{0,2,0}$.
Similarly, making use of Lemma \ref{tech} options (\ref{ab0}), (\ref{a30}), we get that elements of the form $f_{a,b,0}$, $b\geq3$,
are included in the $K[X_3]^{S_3}$-module generated by $\{f_{a,1,0},f_{a,2,0},f_{0,a,0}\}$, or by $\{f_{1,a,0},f_{2,a,0},f_{0,a,0}\}$ via Lemma \ref{tech} (\ref{abba}).
Note that $f_{1,1,0}=f_{2,2,0}=0$, $f_{2,1,0}=-f_{1,2,0}$, and
elements of the form $f_{0,a,0}$ are in the $K[X_3]^{S_3}$-module generated by $\{f_{0,1,0},f_{0,2,0}\}$ when $a\geq3$.
Hence the elements $f_{1,a,0}$ and $f_{2,a,0}$, $a\geq3$, are included in the
$K[X_3]^{S_3}$-module generated by the elements $f_{1,2,0}$, $f_{0,1,0}$ and $f_{0,2,0}$ using Lemma  \ref{tech} options (\ref{0b0}), (\ref{ab0}), and (\ref{a30}),
inductively. 
\end{proof}

\begin{theorem}\label{main}
The set $\{f_{0,1,0},f_{0,2,0},f_{1,2,0}\}$ is a minimal generating set for the $K[X_3]^{S_3}$-module $(F'_3)^{S_3}$.
Moreover, $(F'_3)^{S_3}$ is freely generated by the elements $f_{0,1,0}$, $f_{0,2,0}$, and $f_{1,2,0}$ as a $K[X_3]^{S_3}$-module.
\end{theorem}

\begin{proof}
Recall that $f_{0,1,0}$, $f_{0,2,0}$, and $f_{1,2,0}$ are of homogeneous degree $3$, $4$, and $5$, in the graded algebra $F_3'$, of the form
\[
f_{0,1,0}=(x_2-x_1)[x_2,x_1]+(x_3-x_1)[x_3,x_1]+(x_3-x_2)[x_3,x_2]
\]
\[
f_{0,2,0}=(x_2^2-x_1^2)[x_2,x_1]+(x_3^2-x_1^2)[x_3,x_1]+(x_3^2-x_2^2)[x_3,x_2]
\]
\[
f_{1,2,0}=(x_1x_2^2-x_1^2x_2)[x_2,x_1]+(x_1x_3^2-x_1^2x_3)[x_3,x_1]+(x_2x_3^2-x_2^2x_3)[x_3,x_2]
\]

\noindent Clearly the element $f_{0,1,0}$ of degree $3$ is not included in the $K[X_3]^{S_3}$-module generated by $\{f_{0,2,0},f_{1,2,0}\}$.
On the other hand, assuming that $f_{0,2,0}$ is  in the $K[X_3]^{S_3}$-module generated by $\{f_{0,1,0},f_{1,2,0}\}$, we have 
$f_{0,2,0}=\alpha\sigma_1f_{0,1,0}$
for some $\alpha\in K$. Now by Corollary \ref{free}, considering the coefficient of $[x_2,x_1]$
in the free $K[X_3]$-module $F_3'$ generated by $[x_2,x_1]$, $[x_3,x_1]$, and $[x_3,x_2]$,  we have that 
\[
(x_2^2-x_1^2)=\alpha(x_1+x_2+x_3)(x_2-x_1)
\]
and hence $\alpha=0$, which contradicts with $f_{0,2,0}\neq0$.
Besides, if $f_{1,2,0}$ is included in the $K[X_3]^{S_3}$-module generated by $\{f_{0,1,0},f_{0,2,0}\}$, then 
\[
f_{1,2,0}=(\alpha\sigma_1^2+\beta\sigma_2)f_{0,1,0}+\gamma\sigma_1f_{0,2,0}
\]
for some $\alpha,\beta,\gamma\in K$. Similarly, let us consider the coefficient
\begin{align}
(x_1x_2^2-x_1^2x_2)=&\left(\alpha(x_1+x_2+x_3)^2+\beta(x_1x_2+x_1x_3+x_2x_3)\right)(x_2-x_1)\nonumber\\
&+\gamma(x_1+x_2+x_3)(x_2^2-x_1^2)\nonumber
\end{align}
of $[x_2,x_1]$ in the equation. Clearly $\alpha x_3^2x_2=0$ gives $\alpha=0$, and then by coefficients of $x_1x_2^2$ and $x_2^2x_3$ we have
$\beta+\gamma=1$ and $\beta+\gamma=0$, which is a contradiction. Therefore, the set $\{f_{0,1,0},f_{0,2,0},f_{1,2,0}\}$ is a minimal generating set.

Finally, we have to show that $f_{0,1,0}$, $f_{0,2,0}$, and $ f_{1,2,0}$ are free generators.
Note that $f_{0,1,0}^2=f_{0,2,0}^2=f_{1,2,0}^2=0$ due to the equation (\ref{theorem}), and thus we may assume that these elements can be taken linearly in a possible relation.
Let 
\[
s_1f_{0,1,0}+s_2f_{0,2,0}+s_3f_{1,2,0}=0
\]
for some symmetric polynomials $s_1,s_2,s_3\in K[X_3]^{S_3}$. In the same way, by the coefficient of $[x_2,x_1]$, we obtain that
\[
s_1(x_2-x_1)+s_2(x_2^2-x_1^2)+s_3(x_1x_2^2-x_1^2x_2)=0
\]
or
\[
s_1+s_2(x_2+x_1)+s_3x_1x_2=0
\]
in the commutative polynomial algebra $K[X_3]$. This means that the polynomial 
\[
s_2(x_2+x_1)+s_3x_1x_2
\]
is symmetric, which is preserved under the action of the transposition $(23)\in S_3$. Then we have the following.
\begin{align}
s_2(x_2+x_1)+s_3x_1x_2&=s_2(x_3+x_1)+s_3x_1x_3\nonumber\\
s_2x_2+s_3x_1x_2&=s_2x_3+s_3x_1x_3\nonumber\\
x_2(s_2+s_3x_1)&=x_3(s_2+s_3x_1)\nonumber
\end{align}
Therefore, $(x_2-x_3)(s_2+s_3x_1)=0$, implying that $s_2+s_3x_1=0$. This implies that $s_3x_1$ is symmetric.
However, $s_3x_1=(12)\cdot s_3x_1=s_3x_2$ gives $s_3=0$, where $(12)\in S_3$. Hence $s_2=s_3x_1=0$, and $s_1=-s_2(x_2+x_1)-s_3x_1x_2=0$.
\end{proof}

We give the next example to illustrate how to express a symmetric polynomial in terms of free generators $f_{0,1,0}$, $f_{0,2,0}$, and $f_{1,2,0}$.

\begin{example} Let us consider the symmetric polynomial
\[
f_{2,4,5}=(x_1^2x_2^4-x_1^4x_2^2)x_3^5[x_2,x_1]+(x_1^2x_3^4-x_1^4x_3^2)x_2^5[x_3,x_1]+(x_2^2x_3^4-x_2^4x_3^2)x_1^5[x_3,x_2].
\]
 Then we have that $f_{2,4,5}=\sigma_3^2f_{0,2,3}$ by Case II of the proof of Theorem \ref{main}.
Now by 
Lemma \ref{tech} (\ref{0bc}),
\[
f_{2,4,5}=\sigma_3^2(\nu_3f_{0,2,0}-f_{0,5,0}+f_{2,3,0})
\]
and by options (\ref{0b0}), (\ref{a30}) and (\ref{abba}) of  Lemma \ref{tech}, we have that
\begin{align*}
f_{2,4,5}&
=\sigma_3^2(\nu_3f_{0,2,0}-(\sigma_1f_{0,4,0}-\sigma_2f_{0,3,0}+\sigma_3f_{0,2,0})+(\sigma_1f_{2,2,0}-\sigma_2f_{2,1,0}-\sigma_3f_{0,2,0}))\\&
=\sigma_3^2(\nu_3f_{0,2,0}-(\sigma_1(\sigma_1f_{0,3,0}-\sigma_2f_{0,2,0}+\sigma_3f_{0,1,0})-\sigma_2f_{0,3,0}+\sigma_3f_{0,2,0})\\&\ \ \ \ \ \ \ \
+(\sigma_2f_{1,2,0}-\sigma_3f_{0,2,0}))\\& 
=\sigma_3^2(-\sigma_1\sigma_3f_{0,1,0}+(\nu_3+\sigma_1\sigma_2-2\sigma_3)f_{0,2,0}+(-\sigma_1^2+\sigma_2)f_{0,3,0}+\sigma_2f_{1,2,0})\\&
=\sigma_3^2(-\sigma_1\sigma_3f_{0,1,0}+(\nu_3+\sigma_1\sigma_2-2\sigma_3)f_{0,2,0}+(-\sigma_1^2+\sigma_2)(\sigma_1f_{0,2,0}-\sigma_2f_{0,1,0})\\&\ \ \ \ \ \ \ \ +\sigma_2f_{1,2,0})\\&
=(-\sigma_1\sigma_3^3+\sigma_1^2\sigma_2\sigma_3^2-\sigma_2^2\sigma_3^2)f_{0,1,0}\\& \ \ \ +(\nu_3\sigma_3^2+\sigma_1\sigma_2\sigma_3^2-2\sigma_3^3-\sigma_1^3\sigma_3^2+\sigma_1\sigma_2\sigma_3^2)f_{0,2,0}+\sigma_2\sigma_3^2f_{1,2,0}.
\end{align*}
\end{example}

\noindent We complete the paper by the next problem.

\begin{problem}
Determine symmetric polynomials in $F'_n$ for $n\geq4$.
\end{problem}

\end{document}